\documentclass[a4paper,twoside, reqno]{amsart}
\usepackage{mathrsfs}
\usepackage{amsfonts}
\usepackage{amssymb}
\usepackage{amsxtra}
\usepackage{dsfont}
\usepackage{amsthm}
\usepackage[colorlinks=true,linkcolor=blue]{hyperref}%
\usepackage[T1]{fontenc}
\usepackage[utf8]{inputenc}
\usepackage[english]{babel}
\usepackage{amsmath} 
\usepackage{bm}
\usepackage{graphicx}
\usepackage{hyperref}
\usepackage{color}
\usepackage{xcolor}
\usepackage{relsize}
\usepackage{soul}

\usepackage{hyperref}
\usepackage{stmaryrd}
\usepackage{listings}
\usepackage{yhmath} %wideparen command
\usepackage[shortlabels]{enumitem}%allows resuming enumerate counter

\numberwithin{equation}{section}
\usepackage{url}
\newtheorem{Thm}{Theorem}[section]

\newtheorem{Cor}[Thm]{Corollary}
\newtheorem{Prop}[Thm]{Proposition}
\newtheorem{Lem}[Thm]{Lemma}
\theoremstyle{definition}%boldface title, romand body
\newtheorem{Rem}{Remark}      
\theoremstyle{definition}
\newtheorem{Defn}[Thm]{Definition}
\newtheorem{Ex}[Thm]{Example}
\newtheorem{Fact}[Thm]{Fact}
\theoremstyle{definition}%boldface title, romand body
\newtheorem*{Prob*}{Problem}

\theoremstyle{definition}

\newcommand{\Z}{\mathbb{Z}}

\newcommand{\N}{\mathbb{N}}

\newcommand{\R}{\mathbb{R}}

\DeclareMathOperator{\Ima}{Im}
\DeclareMathOperator{\Rea}{Re}

\begin{document}
%%%%%%%%%%%%%%%%%%%%%%%%%%%%%%%%%%%%%%%%%%%%%%%%%%%%%%%%%%%%%%%%%%%%%%%%%%%%%%%%%%%%%%%%%%%%%%%%%%%%%%%%

\selectlanguage{english}

\title{Lattice points in high-dimensional $\ell^q$ balls with small radii}

\begin{abstract}
We prove a dimension-free count of lattice points in high-dimensional $\ell^q$ balls and spheres with small radii. Our result complements those already established for $q=1,2$ by covering the whole range of integers $q\ge 2$. 
\end{abstract}

\thanks{The author was supported by National Science Centre, Poland, grant Sonata Bis 2022/46/E/ST1/00036.
For the purpose of Open Access the author has applied a CC BY
public copyright license to any Author Accepted Manuscript (AAM) version arising from
this submission.}

\author{Michał Dymowski}
\email{michal.dymowski@math.uni.wroc.pl}

\maketitle

.

\section{Introduction}
This work is a continuation and improvement of the results established in \cite{l2_balls}. In order to state our main result, we first introduce the minimal necessary notation.

For a fixed dimension $d\in\N$ and an integer $q\in\N_{>1}$, let $n^{1/q}S^q$ and $n^{1/q}B^q$ denote the $\ell^q$ sphere and ball, respectively, of radius $n^{1/q}$ in $\R^d$; that is,
\[
n^{1/q}S^q = \Big\{x\in\R^d : \sum_{k=1}^d |x_k|^q = n\Big\},\qquad n^{1/q}B^q = \Big\{x\in\R^d : \sum_{k=1}^d |x_k|^q \leq n\Big\}. 
\]
The primary focus of this paper is to provide a dimension-free formula for computing the number of lattice points inside $\ell^q$ balls and spheres in the small-scale regime. This generalizes the dimension-free lattice point count for Euclidean balls ($q=2$) introduced in \cite[Theorem 3.4]{l2_balls}. The main result of this paper, Theorem \ref{thm:3.4}, also improves the corresponding result for $q=2$ from \cite{l2_balls} by simplifying the formula for the exponent in the estimate \eqref{improvement}. An analogous result for $q=1$ was previously proven in \cite[Theorem 1.4]{l1_balls}. Below, we state an abbreviated version of Theorem \ref{thm:3.4}.

\begin{Thm}[Part of Theorem \ref{thm:3.4}] \label{thm:3.4_partial}
Let $\alpha = n/d$. There exists a universal constant $c\in(0,1)$, independent of $q$ and $d$, and a function $b_q$ holomorphic on the disk $\{z\in\mathbb{C}: |z| \leq c\}$, such that for all $n$ in the range $1\leq n\leq cd$, we have
\begin{equation} \label{improvement}
|n^{1/q}B^q| \approx 2^n \alpha^{-n}  \frac{1}{\sqrt{n}}\exp \big(n b_q(\alpha) \big) \approx |n^{1/q}S^q|.
\end{equation}
Here, the relation $\Gamma\approx\Delta$ for two nonnegative quantities denotes that there exists an absolute constant $C \geq 1$, independent of the dimension $d$ and parameter $q$, such that $C^{-1}\Gamma \leq \Delta \leq C\Gamma$. 
Moreover, the Taylor expansion of $b_q$ around the origin is given by 
\[
   b_q(\alpha)= 1 - \frac{1}{2}\alpha - \frac{1}{6} \alpha^2 + \sum_{k=4}^\infty b_{k,q}\alpha^k,
\]
and further coefficients $b_{k,q}$ can be explicitly computed if needed.
\end{Thm}

One of the primary motivations for establishing these dimension-free lattice point counts arises from the study of maximal averaging operators. Let $G$ be a subset of $\R^d$. For every $t\ge 0$, we let $tG=\{tx\colon x\in\R^d\}$ be the dilation of $G.$ Let $\mathbb{G}\subset [0,\infty)$ be the set of those $t\ge 0$ such that $tG\cap \Z^d$ is non-empty. For $t\in \mathbb{G}$, we consider the discrete Hardy-Littlewood averaging operator
\begin{align*}
	\mathcal M_t^Gf(x):=\frac{1}{|tG\cap \Z^d|}\sum_{y\in tG\cap\Z^d}f(x-y),
	 \qquad f\in \ell^1(\Z^d),\quad 
	x\in\Z^d.
\end{align*}
The symbol $|tG\cap \Z^d|$ above stands for the number of elements in the set $tG\cap \Z^d.$

\begin{Rem}
It is possible to obtain dimension-free estimates for the maximal functions $\sup_{t\geq 0}|\mathcal{M}^{B^q}_tf|$ and $\sup_{t\geq 0} |\mathcal{M}^{S^q}_t f|$ for $q\geq 2$ in the small-scale regime $t < d^{1/q - \varepsilon}$ by combining the techniques used in the proof of \cite[Theorem 1.2]{l2_balls} with Theorem \ref{thm:3.4}. We do not provide the full argument here, as a more general version of this result is given in \cite{niks_maksymalne}.
\end{Rem}

\begin{Rem}
It may be interesting to note that if we consider only $q > K$ for some constant $K>0$, then the first $2^K-1$ coefficients of $b_q$ are independent of $q$.
\end{Rem}

\subsection{Notation}
\label{ssec:not}
We finish the introduction with a description of our notation. 

\begin{enumerate}
\item For an integer $k>0$, we use the notation $\N_{\geq k}= \{k,k+1,... \}$.
\item To denote $\ell^q$ spheres and balls of radius $n^{1/q}$ in $\R^d$, we use the notation introduced in Theorem \ref{thm:3.4_partial}, that is
\begin{align*}
n^{1/q}S^q &= n^{1/q}S^q(d) = \{x\in\R^d : \sum_{k=1}^d |x_k|^q = n\}, \\
n^{1/q}B^q &= n^{1/q}B^q(d) = \{x\in\R^d : \sum_{k=1}^d |x_k|^q \leq n\}.
\end{align*}
 \item We also use the abbreviated notation $|n^{1/q}S^q|, |n^{1/q}B^q|$ for the number of lattice points inside those sets.
\item Throughout the paper, the letter $d\in \N$ is reserved for the dimension and all implicit constants will be
independent of $d$.  
\item For two nonnegative quantities $X$ and $Y$,
we write $X \lesssim_{\delta} Y$ if there is an absolute constant
$C_{\delta}>0$ depending only on $\delta>0$ such that $X\le C_{\delta}Y$.
We  write $X \approx_{\delta} Y$ when
$X \lesssim_{\delta} Y\lesssim_{\delta} X$. We will omit the subscript
$\delta$ if the implicit constants are universal.
\item The above convention is also used for Big-$O$ notation. In particular $X=O_{\delta}(Y)$ means that $|X|\lesssim_{\delta} |Y|.$
\end{enumerate}

\subsection*{Acknowledgements}
The author wishes to thank Błażej Wróbel for introducing him to this research topic and the results of \cite{l2_balls}

\section{Statement of the result}
\label{sec:3}
In this section we provide the full statement  of Theorem \ref{thm:3.4}. It closely follows the proof of \cite[Theorem 3.4]{l2_balls} but involves bounding various quantities related to the function $h_q$ (defined below), independently of the parameter $q$.
This is required to obtain a universal constant $c$ (independent of the parameters $d$ and $q$) in the statement of Theorem \ref{thm:3.4}.

\begin{Defn}
    \label{def:3.3}
For every $q\in\N_{\geq 2}$ we define the function $h_q : \{ z \in \mathbb{C}: |z|<1 \} \to \mathbb{C}$ by the equation
\[
h_q(z)= \sum_{k \in \Z} z^{|k|^q}=1+2 \sum_{k=1}^{\infty} z^{|k|^q}.
\]
\end{Defn}
\noindent For every $d \in \N$, the following identities hold:

\begin{gather*}
h_q(z)^d 
= \sum_{n=0}^{\infty} |n^{1/q}S^q | z^n, \\
\frac{h_q(z)^d}{1-z}
= \sum_{n=0}^\infty \Big( \sum_{k=0}^n \big|\sqrt{k}S \big| \Big)z^n= \sum_{n=0}^{\infty} \big|n^{1/q}B^q\big| z^n.
\end{gather*}
Also, for every $r\in(0,1)$, we obtain
 
\begin{equation}
\label{eq:Cauchyball}
|n^{1/q}S^q|= \frac{1}{2 \pi i} \oint_{|z|=r} \frac{h_q(z)^d}{z^{n+1}} dz,\qquad|n^{1/q}B^q|= \frac{1}{2 \pi i} \oint_{|z|=r} \frac{h_q(z)^d}{z^{n+1}}\frac{1}{1-z} dz,
\end{equation}
where the last two equalities above follow from the application of Cauchy's integral formula.

Throughout the remainder of this paper, we assume $n\lesssim d$; consequently, the parameter
\[ \alpha=\frac{n}{d}\]
is bounded by $\alpha\lesssim 1.$ 
\\
We are now ready to state our main result. Theorem \ref{thm:3.4} gives a quantitative and uniform (dimension-free) asymptotic formula for the number of lattice points contained in $\ell^q$ balls and spheres within the regime $1 \leq n \leq c d.$

\begin{Thm}
\label{thm:3.4} There exists $c \in (0,1)$  such that for  $n,d \in \N$ satisfying  $1 \leq n \leq c d$ and $q\in\N_{\geq 2}$ we have
\begin{equation}
\label{eq:{thm:3.4}:hform}
|n^{1/q}B^q| \approx \frac{h_q(r)^d}{r^n} \frac{1}{\sqrt{n}} \approx |n^{1/q}S^q|,
\end{equation}
where $r \in (0,1)$ is the unique number satisfying $r \frac{h'_q(r)}{h_q(r)}=\alpha$ and the implicit constants do not depend on $n$, $d$ and $q$.

Moreover, for every $q\in\N_{\geq 2}$ there exists a function $b_q$ holomorphic on $|\alpha|< (1/64)^3$ such that for any $n,d \in \N$ satisfying $1 \leq n \leq cd$ we have
\begin{equation}
\label{eq:{thm:3.4}:expform}
|n^{1/q}B^q| \approx 2^n \alpha^{-n}  \frac{1}{\sqrt{n}
}\exp \Big(nb_q(\alpha) \Big) \approx |n^{1/q}S^q|,
\end{equation}
where
\[
b_q(\alpha) = 1 - \sum_{k=1}^{2^q-2}\frac{1}{k(k+1)}\alpha^k + \left(\frac{1}{2^{2^q-1}} - \frac{1}{2^q(2^q-1)}\right)\alpha^{2^q-1} + \sum_{k=2^q}^\infty b_{k,q}\alpha^k.
\]

\end{Thm}
\begin{Rem}
The quantities $|n^{1/q}B^q|$ and $|n^{1/q}S^q|$ can be approximated reasonably well as long as we fix a constant $K>0$ and consider only the scales satisfying $n \alpha^{K+1} \leq 1$ (that is, $n \leq d^{\frac{K+1}{K+2}})$. Then obtaining the approximation requires computing the coefficients $(b_{k,q})_{k \leq K}$ which could be tedious but can be done.
\end{Rem}

\section{Proof of the result}
The proof of Theorem \ref{thm:3.4} relies on the idea of the saddle point method explained in \cite[Section 3]{MO}. We start by proving that for each pair of parameters $(\alpha, q)$ the value $r$ mentioned in \eqref{eq:{thm:3.4}:hform} is unique. A similar argument for the $q=2$ case was already present in \cite{MO}.

\begin{Lem}
\label{lem:3.5}
For each pair $(\alpha, q)$, where $0<\alpha<1$ and $q\in\N_{\geq 2}$, there exists a unique $r \in (0,1)$ such that 
\[
r \frac{h'_q(r)}{h_q(r)}=\alpha.
\]
\end{Lem}

\begin{proof}
Note that $\lim_{z \to 0^+} z \frac{h'_q(z)}{h_q(z)} = 0$ and $\lim_{z \to 1^-} z \frac{h'_q(z)}{h_q(z)} = \infty$. Thus, it is sufficient to show that the function $z \mapsto z\frac{h'_q(z)}{h_q(z)}$ is strictly increasing in $z$.

From the identity
\[
\frac{d}{dz}\left(z\frac{h'_q(z)}{h_q(z)}\right) = \frac{h'_q(z)}{h_q(z)} + \frac{z h''_q(z)}{h_q(z)} - \frac{z (h'_q(z))^2}{(h_q(z))^2},
\]
it follows that this is equivalent to the condition
\[
h_q(z)\left(z h'_q(z) + z^2 h''_q(z)\right) > z^2 (h'_q(z))^2, \quad \text{for } z\in (0,1).
\]

Using the series expansion $h_q(z) = \sum_{k\in\Z} z^{k^q}$, this can be rewritten as
\[
\left( \sum_{k\in\Z} z^{k^q} \right) \left( \sum_{k\in\Z} k^{2q}z^{k^q} \right) > \left( \sum_{k\in\Z} k^q z^{k^q}\right)^2.
\]

The last inequality is a direct application of the Cauchy-Schwarz inequality. Since $k^{2q}z^{k^q}$ grows much faster than $z^{k^q}$ with $k$, the sequences in the sums are not proportional, which means the inequality has to be strict.
\end{proof}

To this point, we have not established any quantitative relationship between $r$ and $\alpha$ which appear in Lemma \ref{lem:3.5}.
Such a result will be necessary for the subsequent application of the saddle point method, so we derive it here. 
From the proof it may seem that the radius of convergence $2c_1$ depends on the parameter $q$ and could potentially tend to 0 as $q$ approaches infinity, but Remark \ref{rem_r_function} explains that it can be bounded from below regardless of $q$.

\begin{Prop}
 \label{thm:3.6}
There exists $c_1 >0$ such that for all $q\in\N_{\geq2}$ and all  $|\alpha|<c_1$,  the equation
\[
z \frac{h'_q(z)}{h_q(z)}= \alpha
\]
has a unique solution $r$ satisfying $r\in (0,1).$ The solution is given by the power series
\begin{equation} \label{r_function}
r= \sum_{k=1}^{\infty} a_{k,q} \alpha^k,
\end{equation}
with a radius of convergence of at least $2c_1$, where
\[
a_{k,q}= \frac{1}{k!} \Big\{  \frac{d^{k-1}}{dz^{k-1}} \Big(  \frac{h_q(z)}{h'_q(z)}\Big)^k \Big\}_{z=0},
\]
in particular
\[
a_{1,q}=a_{2,q}=...=a_{2^q-1,q}=\frac{1}{2},\qquad  a_{2^q,q}= \frac{1}{2} - 2^{q-2^q}.
\]
\end{Prop}
\begin{proof}
Denote $H_q(r)=rh'_q(r)/h_q(r)$. Since $H'_q(0)=2\neq 0$, from the implicit function theorem we get that the mapping $H_q(r)\mapsto r$ is analytic in a small neighbourhood of 0.
Therefore, for sufficiently small $\alpha$, we can apply the classical Lagrange-Bürmann inversion theorem to obtain the formula for the coefficients $a_{k,q}$ from the equation $H_q(r) = \alpha$.
\end{proof}

\begin{Rem} \label{rem_r_function}
    The radius of convergence in the equation \ref{r_function} and the coefficients $a_{k,q}$ can be bounded independently of $q$.
    For $|z| < 1/16$ we get
    \begin{align*}
    |h'_q(z)| &= 2\left|1 + \sum_{k=2}^{\infty} k^{q}z^{k^q - 1}\right| > 2\left|1 - \sum_{k=2}^{\infty} k^{q}|z|^{k^q - 1}\right| > 2\left| 1 - \sum_{k=1}^{\infty}(k+1)|z|^k \right| \\
    &= 2\left| 1 - \frac{|z|}{1-|z|}(1 + \frac{1}{1-|z|})\right| > 1,
    \end{align*}
    where the last inequality on the first line comes from the fact that every element of the form $k^q|z|^{k^q-1}$ is included in the sum  $\sum_{k}(k+1)|z|^k$.
    Consequently, the function $\frac{h_q(z)}{h'_q(z)}$ is analytic on some neighbourhood of $|z|\leq 1/16$.
    An application of the Cauchy integral formula yields
    \begin{align*}
        |a_{k,q}|&= \frac{1}{k!} \left| \Big\{  \frac{d^{k-1}}{dz^{k-1}} \Big(  \frac{h_q(z)}{h'_q(z)}\Big)^k \Big\}_{z=0} \right| \\
        &= \frac{1}{2\pi k} \left| \int_{|\cdot| = 1/16} \Big(  \frac{h_q(z)}{h'_q(z)}\Big)^k \frac{1}{z^{k+1}} dz \right|
        < \frac{3^k 16^k}{k} < 64^k.
    \end{align*}
    From the bound on $a_{k,q}$ we conclude that for $|\alpha|\leq 1/64$ and any $q\in\N_{\geq 2}$, the function $r=r(\alpha)$ is analytic and given by the formula \ref{r_function}.
    
\end{Rem}

\begin{Rem} \label{rem_r_approx}
Before proceeding to the proof of Theorem \ref{thm:3.4}, we note a useful consequence of Remark \ref{rem_r_function}. For sufficiently small $\alpha$, the solution $r$ admits the expansion
\[
    r = \frac{\alpha}{2} + O(\alpha^2),
\]
uniformly for all $q \in \mathbb{N}_{\geq 2}$. This first-order approximation immediately yields the inequalities
\[
    \frac{\alpha}{4} < r < \alpha < \frac{1}{4}.
\]
\end{Rem}

To implement the saddle point method, we introduce the function 
\begin{equation}
\label{eq:fdefi}
    f(z) = f_q(z) = \log(h_q(z)) - \alpha \log(z),
\end{equation}
which is analytic on a sufficiently small neighborhood of the arc $\{z \in \mathbb{C} : |z|=r, |\arg(z)| \leq \delta\}$, for a small parameter $\delta > 0$ independent of $q$. The exact value of $\delta$ will be determined during the proof of Theorem \ref{thm:3.4}. 
The function $f$ naturally encodes the main part of the integrand for both integrals from the identities in \eqref{eq:Cauchyball}, since 
\[
\exp(df(z)) = \frac{h_q(z)^d}{z^n}.
\]
Recall the definition of $r$, which yields the critical point condition
\[
f'(r) = \frac{h'_q(r)}{h_q(r)} - \frac{\alpha}{r} = 0.
\]
Keeping in mind the condition $r\approx \alpha$ from Remark \ref{rem_r_approx}, we differentiate $f$ once more to obtain the relation
\[
f''(r) = \frac{\alpha}{r^2} + \frac{d}{dr}\left(\frac{h'_q(r)}{h_q(r)}\right) = \frac{\alpha}{r^2} + O(1) \approx \frac{1}{\alpha}.
\]
More precisely, we have $f''(r)\in(\frac{1}{2\alpha}, \frac{5}{\alpha})$. 

Applying Taylor's theorem in the complex plane around $z=r$ along the contour $|z|=r$ with $|\arg(z)| \leq \delta$, we get
\[
    f(z) = f(r) + \frac{\beta}{2}(z-r)^2 + \frac{1}{2} \int_{\wideparen{r,z}} (w-z)^2 f^{(3)}(w) dw,
\]
where $\wideparen{r,z}$ denotes the circular arc connecting $r$ to $z$. 
The third derivative expands as $f^{(3)}(w) = -\frac{2\alpha}{w^3} + \frac{d^2}{dw^2}\left(\frac{h'_q(w)}{h_q(w)}\right)$. 
For $w$ restricted to this arc, $|w| = r \approx \alpha$, which directly implies $|f^{(3)}(w)| \lesssim \frac{\alpha}{r^3} + 1 \lesssim \frac{1}{\alpha^2}$. 
Consequently, the Taylor expansion simplifies to
\begin{equation}
\label{eq:ftay}
    f(z) = f(r) + \frac{f''(r)}{2}(z-r)^2 + O\left(\frac{1}{\alpha^2}|z-r|^3\right),
\end{equation}
where the big-$O$ constant is again uniform in $q$.

\par 

We are now equipped to proceed with the proof of Theorem \ref{thm:3.4}.

\begin{proof}[Proof of Theorem \ref{thm:3.4}]

We first address the regime where $n$ is bounded by some constant $C > 0$. 
In this case, the relation $r \approx \alpha \approx_C 1/d$ holds, meaning our target asymptotic expression behaves like
\[
\frac{h_q(r)^d}{r^n \sqrt{n}} \approx_C d^n.
\]
To verify \eqref{eq:{thm:3.4}:hform} for this range, we can simply rely on elementary combinatorial bounds. 

For the lower bound, we restrict our attention to lattice points with coordinates in $\{-1, 0, 1\}$. It is easy to see that
\[
    |n^{1/q}B^q| \geq |n^{1/q}S^q| \geq \left| \{-1,0,1\}^d \cap n^{1/q}S^q \right| = 2^n \binom{d}{n} \approx_C d^n.
\]
For the upper bound, we use the geometric inclusion $n^{1/q}B^q \subseteq n B^1$, where $B^1$ is the standard $\ell^1$ ball in $\R^d$. 
Using the standard formula for the number of lattice points in $n B^1$ (see, for instance, \cite[Lemma 2.2]{Ni1} for a proof), we obtain
\[
    |n^{1/q}B^q| \leq |n B^1| = \sum_{j=0}^n 2^j \binom{d}{j}\binom{n}{j} \lesssim_C d^n.
\]
Combining these bounds shows that for small scales $n < C$, we trivially have 
\[
|n^{1/q}B^q| \approx_C |n^{1/q}S^q| \approx_C \frac{h_q(r)^d}{r^n \sqrt{n}},
\] 
establishing \eqref{eq:{thm:3.4}:hform}. 

For the remainder of the proof, we assume $n > C$. Our strategy is to establish the upper bound for $|n^{1/q}B^q|$ and the lower bound for $|n^{1/q}S^q|$. 
Throughout the following steps, whenever variables are forced to be "sufficiently small" (for $\alpha, r, \delta$) or "sufficiently large" (for $n$), it means choosing them relative to universal constants independent of $q$. The exact value of the constant $C$ will be determined in the course of the proof.
We introduce absolute non-negative constants $c_2, c_3, c_4, c_5$ as needed in subsequent steps.

\vspace{0.5cm}
\noindent \textbf{1) Upper bound for $|n^{1/q}B^q|$}. 

By Cauchy's integral formula, we can express the number of points in the $\ell^q$ ball as in \eqref{eq:Cauchyball}
\[
    |n^{1/q}B^q| = \frac{1}{2 \pi i} \oint_{|z|=r} \frac{h_q(z)^d}{z^{n+1}} \frac{1}{1-z} dz.
\]
We select a small arc parameter $\delta \in (0, \pi)$ (whose exact size will be uniformly bounded later) and partition the contour integral into a main arc $W_1$ (where $|\arg(z)| \leq \delta$) and an error arc $W_2$ (where $\delta < |\arg(z)| \leq \pi$):
\begin{equation} \label{eq:3.3}
    \frac{1}{2 \pi i} \int_{\substack{|z|=r \\ |\arg(z)| \leq \delta}} \frac{h_q(z)^d}{z^{n+1}} \frac{dz}{1-z} 
    + 
    \frac{1}{2 \pi i} \int_{\substack{|z|=r \\ |\arg(z)| > \delta}} \frac{h_q(z)^d}{z^{n+1}} \frac{dz}{1-z} 
    =: W_1 + W_2.
\end{equation}

We begin with the principal contribution, $W_1$. Substituting the Taylor expansion \eqref{eq:ftay} yields
\begin{align*}
    W_1 &= \frac{1}{2 \pi i} \int_{\substack{|z|=r \\ |\arg(z)| \leq \delta}} e^{d f(z)} \frac{dz}{z(1-z)} \\
        &= \frac{1}{2 \pi i} \int_{\substack{|z|=r \\ |\arg(z)| \leq \delta}} \exp\left(d f(r) + \frac{df''(r)}{2}(z-r)^2 + O\left(\frac{d}{\alpha^2} |z-r|^3\right)\right) \frac{dz}{z(1-z)}.
\end{align*}
Parametrizing the contour via $z = re^{i\theta}$ for $\theta \in [-\delta, \delta]$ and isolating the constant exponential factor, this transforms into
\[
    \frac{h_q(r)^d}{r^n} \frac{1}{2 \pi} \int_{-\delta}^{\delta} \exp\left( \frac{dr^2f''(r)}{2}(1-e^{i\theta})^2 + O\left(\frac{d r^3}{\alpha^2} |1-e^{i\theta}|^3\right) \right) \frac{d\theta}{1-re^{i\theta}}.
\]
Using the approximation $(1-e^{i\theta})^2 = -\theta^2 + O(|\theta|^3)$ and noting that $d r^2 f''(r)  \in (\frac{d\alpha}{32}, 5d\alpha) = (\frac{n}{32}, 5n)$, we obtain
\[
    W_1 = \frac{h_q(r)^d}{r^n} \frac{1}{2 \pi} \int_{-\delta}^{\delta} \exp\left( -\theta^2 \frac{d r^2 f''(r) }{2} (1 + O(\delta)) \right) (1 + O(r)) d\theta.
\]
We fix $\delta \in (0,1)$ and force $\alpha$ (and thus $r$) to be small enough so that the error terms $O(\delta)$ and $O(r)$ are bounded by $1/2$. 
This allows us to compare the integrand to a standard Gaussian:
\[
    |W_1| \leq \frac{h_q(r)^d}{r^n} \frac{9}{8 \pi} \int_{-\delta}^{\delta} \exp\left( -\theta^2 \frac{dr^2f''(r)}{4} \right) d\theta 
    \lesssim \frac{h_q(r)^d}{r^n} \frac{1}{\sqrt{d r^2 f''(r) }} \int_{-\infty}^{\infty} e^{-x^2/2} dx 
    \approx \frac{h_q(r)^d}{r^n} \frac{1}{\sqrt{n}},
\]
where we once again used the property $d r^2 f''(r) \approx n$.

To control the error arc integral $W_2$, let $z = r e^{i\theta}$ with $\theta \in (\delta, \pi]$. Applying the elementary bound $\cos \theta \leq 1 - \theta^2/20$ for $\theta \in [0, \pi]$, we get
\begin{align*}
    |1+2z|^2 &= 1 + 4r^2 + 4r \cos \theta \leq (1+2r)^2 - \frac{r\theta^2}{5} 
    \leq (1+2r)^2 - \frac{r\delta^2}{5} \\
    &= 1 + 2r\left(2 - \frac{\delta^2}{10}\right) + 4r^2 
    \leq \left(1 + 2\left(1 - \frac{\delta^2}{40}\right)r\right)^2.
\end{align*}
Setting the constant $c_2 = \delta^2/40$, this allows us to bound the ratio
\[
    \frac{|h_q(z)|}{h_q(r)} \leq \frac{1 + 2(1-c_2)r + O(r^4)}{1+2r} \leq 1 - \frac{c_2 r + O(r^4)}{1+2r} \leq \exp(-c_2 r / 2),
\]
for sufficiently small $r$. 
Applying this to the integrand results in 
\begin{equation} \label{eq:W2bound}
    |W_2| \lesssim \frac{h_q(r)^d}{r^n} \exp(-d r c_2 / 2) \lesssim \frac{1}{\sqrt{n}} \frac{h_q(r)^d}{r^n},
\end{equation}
where the final inequality holds for large enough $n$ since $dr \approx d\alpha = n$. 
Combining the bounds for $W_1$ and $W_2$, we conclude that
\begin{equation} \label{eq:snBabo}
    |n^{1/q}B^q| \lesssim \frac{1}{\sqrt{n}} \frac{h_q(r)^d}{r^n}.
\end{equation}

\vspace{0.5cm}
\noindent \textbf{2) Lower bound for $|n^{1/q}S^q|$}. 

For the $\ell^q$ spheres, the Cauchy integral formula yields
\[
    |n^{1/q}S^q| = \frac{1}{2 \pi i} \oint_{|z|=r} \frac{h_q(z)^d}{z^{n+1}} dz.
\]
We split this integral into $V_1$ (the main arc, $|\arg(z)| \leq \delta$) and $V_2$ (the error arc, $|\arg(z)| > \delta$) analogously to the $\ell^q$ ball case. 
The contribution of $V_2$ is negligible due to its geometric decay in $n$ by the same argument as in \eqref{eq:W2bound}:
\begin{equation} \label{eq:V2ab}
    |V_2| \lesssim \exp(-c_3 n) \frac{h_q(r)^d}{r^n},
\end{equation}
for some absolute constant $c_3 > 0$.

To bound the central term $V_1$ from below, we consider the contributions from the real and imaginary parts of the function $f$ separately. Let 
\[
    \varphi(\theta) = \Ima(f(re^{i\theta})), \qquad \psi(\theta) = \Rea(f(re^{i\theta}) - f(r)).
\]
Using the identity $\exp(d f(r)) = h_q(r)^d r^{-n}$, we express the real part of $V_1$ as
\begin{equation} \label{eq:ReaV1}
    \Rea(V_1) = \Rea\left( \frac{1}{2 \pi i} \int_{\substack{|z|=r \\ |\arg(z)| \leq \delta}} e^{d f(z)} \frac{dz}{z} \right) = \frac{1}{2\pi} \frac{h_q(r)^d}{r^n} \int_{-\delta}^{\delta} \cos(d\varphi(\theta)) \exp(d\psi(\theta)) d\theta.
\end{equation}
From the Taylor expansion \eqref{eq:ftay}, we observe that the leading imaginary error term grows as $\Ima((r - r e^{i\theta})^2) = O(r^2 \theta^3)$. Consequently,
\[
    |d\varphi(\theta)| \lesssim d(f''(r) r^2 + r^3/\alpha^2)\theta^3 \lesssim n \theta^3, \quad \text{for } |\theta| \leq \delta.
\]
By taking $n > C$ sufficiently large, we can safely absorb the implicit constant to yield $|d\varphi(\theta)| \leq n^{-1/4}(\sqrt{n}\theta)^3$. 
Similarly, the Taylor expansion \eqref{eq:ftay} ensures that the real part can be controlled by $n\theta^2$:
\[
    -c_4 n \theta^2 \leq d\psi(\theta) \leq -c_5 n \theta^2, \quad \text{for } |\theta| \leq \delta,
\]
for some absolute constants $c_4 > c_5 > 0$. 

Applying the change of variables $s = \sqrt{n}\theta$, equation \eqref{eq:ReaV1} transforms to
\[
    \Rea(V_1) = \frac{1}{2\pi\sqrt{n}} \frac{h_q(r)^d}{r^n} \int_{-\delta\sqrt{n}}^{\delta\sqrt{n}} \cos\left(d\varphi\left(\frac{s}{\sqrt{n}}\right)\right) \exp\left(d\psi\left(\frac{s}{\sqrt{n}}\right)\right) ds.
\]
Within the integration limits, we have $|d\varphi(s/\sqrt{n})| \leq n^{-1/4} s^3$ and $-c_4 s^2 \leq d\psi(s/\sqrt{n}) \leq -c_5 s^2$. 
Choosing a large cutoff constant $C > 0$ 
and $n$ large enough to satisfy both $\delta\sqrt{n} > C$ and $n^{-1/4}C^3 \leq \pi/3$, we can securely bound the integral from below:
\begin{align*}
    &\int_{-\delta\sqrt{n}}^{\delta\sqrt{n}} \cos\left(d\varphi\left(\frac{s}{\sqrt{n}}\right)\right) \exp\left(d\psi\left(\frac{s}{\sqrt{n}}\right)\right) ds \\
    &\geq \int_{-C}^C \cos\left(d\varphi\left(\frac{s}{\sqrt{n}}\right)\right) \exp\left(d\psi\left(\frac{s}{\sqrt{n}}\right)\right) ds - \int_{\delta\sqrt{n} > |s| > C} \exp(-c_5 s^2) ds \\
    &\geq \frac{1}{2} \int_{-C}^C \exp(-c_4 s^2) ds - \int_{|s| > C} \exp(-c_5 s^2) ds \gtrsim 1,
\end{align*}
where the final lower bound holds by explicitly fixing $C$ to be large enough. 
This establishes that $\Rea(V_1) \gtrsim \frac{1}{\sqrt{n}} \frac{h_q(r)^d}{r^n}$. 
Combining this with the negligible contribution $V_2$ coming from the error arc, we successfully obtain the lower bound for large $n$:

\begin{equation} \label{eq:snSbel}
    |n^{1/q}S^q| \gtrsim \frac{1}{\sqrt{n}} \frac{h_q(r)^d}{r^n}.
\end{equation}
Together, \eqref{eq:snSbel} and \eqref{eq:snBabo} conclude the proof of the asymptotic relation \eqref{eq:{thm:3.4}:hform}.

\noindent 
{\bf 3) Verification of \eqref{eq:{thm:3.4}:expform}.} It remains to establish \eqref{eq:{thm:3.4}:expform}. 

We write
    \begin{equation}
    \label{eq:thm:3.4expformproof}
    \begin{split}
    \frac{h_q(r)^d}{r^n}
    &= \exp \Big( d\log(h_q(r))- n \log(r) \Big)
  \\
    &= 2^n \alpha^{-n} \exp\Big(n\Big(\alpha^{-1} \log(h_q(r)) - \log(2r/\alpha)\Big) \Big).
    \end{split}
\end{equation}
From Remark \ref{rem_r_function}, we know that the mapping $\alpha \mapsto 2r/\alpha = 1 +2\sum_{k=1}^{\infty} a_{k+1,q}\alpha^k$ is holomorphic inside the disk $|\alpha| < 1/64$. For $|\alpha| < (1/64)^3$ we get
\[
\left|\frac{2r}{a}-1\right| < 2\sum_{k=1} 64^{k+1}(1/64)^{3k} < 1.
\]
Consequently, the map $\alpha\mapsto\log(2r/\alpha)$ is holomorphic on $|\alpha|<(1/64)^3$.
\\
It remains to show that $\alpha\mapsto \alpha^{-1} \log(h_q(r))$ is holomorphic for small $\alpha$.
From Lemma \ref{lem:3.5} we know that $r\in(0,1)$. Since $h_q(z)$ is a power series with a radius of convergence of 1, the mapping
$\alpha\mapsto h_q(r)$ is holomorphic for $|\alpha|\leq 1/64.$ For the range $|\alpha| < (1/64)^3$ we already obtained $|\frac{2r}{\alpha}-1| < 1$, which implies $r < \alpha.$ From this estimate we obtain
\[
|h_q(r)-1| = 2|\sum_{k=1}^{\infty}r^{k^q}| < 2\sum_{k=1}^{\infty} |\alpha|^k < 1.
\]
It follows that $\alpha\mapsto \alpha^{-1} \log(h_q(r))$ is holomorphic for $0<|\alpha|< (1/64)^3$.
It remains to prove the continuity of this mapping at 0. For small $\alpha$, this function behaves as follows:
\[
 \alpha^{-1} \log(h_q(r))) = \alpha^{-1}\log(2r + O(r^{2^q})) = \alpha^{-1}(2r-1 + O(\alpha^{2})) = 1 + O(\alpha).
\]

    This completes the proof of Theorem \ref{thm:3.4}.

\end{proof}

\end{document}